\documentclass[10pt]{scrartcl}

\usepackage{amsmath, amssymb, amsthm, color, url, fullpage, supertabular}
\usepackage{hyperref}
\usepackage[neverdecrease]{paralist}

\usepackage{natbib}
\setcitestyle{round}
\newcommand*{\arXiv}[1]{\bgroup\color{blue}\href{http://arxiv.org/abs/#1}{arXiv:#1}\egroup}
\newcommand*{\doi}[1]{\bgroup\color{blue}\href{http://dx.doi.org/#1}{doi:#1}\egroup}
\newcommand*{\email}[1]{\bgroup\color{blue}\href{mailto:#1}{#1}\egroup}
\renewcommand*{\url}[1]{\bgroup\color{blue}\href{#1}{#1}\egroup}
\newcommand{\todo}[1]{\bgroup\color{red}\bfseries#1\egroup}

\newcommand*{\ppara}[1]{\noindent\textbf{\textsf{#1}}\,\,}

\usepackage[usenames,dvipsnames]{pstricks}
\usepackage{pst-plot}
\usepackage[margin=0.1em]{subcaption}

\newcommand{\COMMENT}[1]{}
\usepackage{mathtools}
\newcommand*{\bbeta}{\underline{\beta}}
\newcommand*{\ggamma}{\underline{\gamma}}
\newcommand*{\ddelta}{\underline{\delta}}
\newcommand*{\Ball}{\mathbb{B}}
\newcommand*{\Cauchy}{\mathcal{C}}
\newcommand*{\defeq}{\coloneqq}
\newcommand*{\defterm}[1]{\textbf{#1}}
\newcommand*{\disteq}{\ensuremath{\mathrel{\stackrel{\mathrm{d}}{=}}}}
\newcommand*{\E}{\mathbb{E}}
\newcommand*{\Hellinger}{\ensuremath{d_{\mathrm{H}}}}
\newcommand*{\N}{\mathbb{N}}
\newcommand*{\Normal}{\mathcal{N}}
\newcommand*{\one}{\mathbf{1}}
\newcommand*{\quark}{\setbox0\hbox{$x$}\hbox to\wd0{\hss$\cdot$\hss}}
\newcommand*{\rd}{\mathrm{d}}
\newcommand*{\R}{\mathbb{R}}
\newcommand*{\sign}{\mathop{\mathrm{sgn}}\nolimits}
\newcommand*{\Stable}{\mathcal{S}}

\newcommand*{\UU}{\mathcal{U}}
\newcommand*{\YY}{\mathcal{Y}}

\newtheorem{theorem}{Theorem}[section]
\newtheorem{proposition}[theorem]{Proposition}

\theoremstyle{definition}
\newtheorem{definition}[theorem]{Definition}
\newtheorem{remark}[theorem]{Remark}
\newtheorem{example}[theorem]{Example}
\newtheorem{assumption}[theorem]{Assumption}

\makeatletter
\@addtoreset{equation}{section}
\@addtoreset{figure}{section}
\@addtoreset{table}{section}
\makeatother

\begin{document}

\title{Well-posed Bayesian inverse problems and heavy-tailed stable quasi-Banach space priors\footnote{To appear in \emph{Inverse Problems and Imaging}.  This preprint differs from the final published version in layout and typographical details.}}

\author{%
	T.\ J.\ Sullivan\footnote{Institute of Mathematics, Free University of Berlin, and Zuse Institute Berlin, Takustrasse 7, 14195 Berlin, Germany, \email{sullivan@zib.de}} %
}

\date{\today}

\maketitle

\begin{abstract}
	\ppara{Abstract:}
	This article extends the framework of Bayesian inverse problems in infinite-dimensional parameter spaces, as advocated by Stuart (\textit{Acta Numer.}\ 19:451--559, 2010) and others, to the case of a heavy-tailed prior measure in the family of stable distributions, such as an infinite-dimensional Cauchy distribution, for which polynomial moments are infinite or undefined.
	It is shown that analogues of the Karhunen--{Lo\`eve} expansion for square-integrable random variables can be used to sample such measures on quasi-Banach spaces.
	Furthermore, under weaker regularity assumptions than those used to date, the Bayesian posterior measure is shown to depend Lipschitz continuously in the Hellinger metric upon perturbations of the misfit function and observed data.
	\smallskip
	
	\ppara{Keywords:} Bayesian inverse problems, heavy-tailed distribution, Karhunen--{Lo\`eve} expansion, quasi-Banach spaces, stable distribution, uncertainty quantification, well-posedness
	
	\smallskip
	
	\ppara{2010 Mathematics Subject Classification:} 
	65J22 
	(35R30, 
	60E07, 
	62F15, 
	62G35, 
	60B11, 
	28C20)
\end{abstract}

\section{Introduction}
\label{sec:introduction}

The Bayesian perspective on inverse problems has attracted much mathematical attention in recent years \citep{KaipioSomersalo:2005, Stuart:2010}.
Particular attention has been paid to Bayesian inverse problems (BIPs) in which the parameter to be inferred lies in an infinite-dimensional space $\UU$, a typical example being a scalar or tensor field coupled to some observed data via an ordinary or partial differential equation.
Numerical solution of such infinite-dimensional BIPs must necessarily be performed in an approximate manner on a finite-dimensional subspace, but it is profitable to delay discretisation to the last possible moment and consider the original infinite-dimensional problem as the primary object of study, since infinite-dimensional well-posedness results and algorithms descend to any finite-dimensional subspace in a discretisation-independent way, whereas careless early discretisation may lead to a sequence of well-posed finite-dimensional BIPs or algorithms whose stability properties degenerate as the discretisation dimension increases.
Well-posedness results for Banach $\UU$ have been established for infinite-dimensional Gaussian priors by \citet{Stuart:2010}, for Besov priors by \citet{DashtiHarrisStuart:2012}, and for log-concave priors with exponentially thin tails by \citet{HosseiniNigam:2017}.
There is a parallel approach of discretisation invariance, introduced by Lehtinen in the 1990s and advanced by e.g.\ \citet{LassasSaksmanSiltanen:2009}, in which the finite-dimensional BIP is the primary object, but care is taken to ensure the existence of a well-defined continuum limit independent of the discretisation.
A common assumption in these works is some exponential integrability of the prior, and one purpose of this article is to relax this by permitting the prior to be heavy-tailed in the sense of only having finite polynomial moments of order $0 \leq p < \alpha$ for some $\alpha < \infty$, and to explicitly identify the growth rates in the misfit potential that are permissible in such a setting.
This article also permits $\UU$ to be only a \emph{quasi}-normed complete space, i.e.\ a quasi-Banach space.

A prototypical heavy-tailed prior on $\R$ is the Cauchy distribution with location $\delta \in \R$ and width $\gamma > 0$, here denoted $\Cauchy(\delta, \gamma)$, which has the Lebesgue density
\begin{equation}
	\label{eq:Cauchy_density}
	\frac{\rd \Cauchy(\delta, \gamma)}{\rd u} (u) = \frac{1}{\gamma \pi} \frac{1}{1 + ((u - \delta) / \gamma)^{2}} .
\end{equation}
$\Cauchy(\delta, \gamma)$ arises straightforwardly as the distribution of the ratio of two independent Gaussian random variables:
\begin{equation}
	\label{eq:quotient}
	\delta + \frac{x}{z} \sim \Cauchy(\delta, \gamma) \text{ when } x \sim \Normal(0, \gamma^{2}), z \sim \Normal(0, 1) \text{ are independent}.
\end{equation}
$\Cauchy(\delta, \gamma)$ has no well-defined mean, even though it is `obviously' centred on $\delta$, nor indeed polynomial moments of any order greater than $\alpha = 1$.
Despite this, the Cauchy distribution arises naturally in even quite elementary applications.
For example, Cauchy distributions arise naturally from quotients of Gaussian random variables, as in \eqref{eq:quotient}.
More geometrically, if uniform measure on a circle is projected radially onto any line not passing through the centre of the circle, as in Figure \ref{fig:radial}, then the image measure is Cauchy.
\citet{MarkkanenRoininenHuttunenLasanen:2016} have recently reported numerical results on the use of heavy-tailed priors for edge-preserving Bayesian inversion in X-ray tomography, where the seemingly natural choice of a total variation regularisation term cannot be interpreted as a discretisation-invariant Bayesian prior \citep{LassasSiltanen:2004}.

In a Bayesian context, the use of a heavy-tailed prior model in preference to one with exponentially small tails corresponds to a prior belief that large deviations are not exponentially rare events.
For example, in a wavelet basis of $L^{2}([0, 1], \rd x)$, it is not rare to draw samples with localised large deviations (see Figure \ref{fig:wavelets});
physically, these might correspond to inclusions in an otherwise relatively homogeneous material matrix, or edges in a piecewise smooth image.
The asymmetry between the two models is starkly illustrated the following information-theoretic calculation of the Kullback--Leibler divergences (relative entropy distances) between a standard normal and a standard Cauchy distribution on $\R$:
\[
	D_{\textup{KL}} \bigl( \Normal(0, 1) \big\| \Cauchy(0, 1) \bigr) \approx 0.2592 < \infty = D_{\textup{KL}} \bigl( \Cauchy(0, 1) \big\| \Normal(0, 1) \bigr) .
\]
Thus, the approximation of a heavy-tailed Cauchy prior by a thin-tailed Gaussian prior represents an infinite loss of information.
However, asymmetrically, the `defensive' adoption of a Cauchy prior in place of a Gaussian one represents a mild loss of information, with which one gains access to large deviations that would be exponentially rare in the Gaussian model.

The family of stable distributions generalises both the Cauchy and Gaussian examples.
Because the stable family is, by definition, closed under linear combinations of independent members, it is an attractive model for spatial or temporal phenomena that decompose in an additive way over disjoint subsets of space or time.
So, for example, a stable distribution is a natural modelling choice for the net external forces imparted on a passive tracer particle in some medium over a time interval:
a Gaussian model leads to Brownian motion, whereas other stable models lead to L\'evy flights.

Thus, after establishing some background and notation in Section \ref{sec:notation}, the purpose of this article is twofold:

Section \ref{sec:sampling} shows how to define quasi-Banach space analogues of heavy-tailed stable distributions via Karhunen--{Lo\`eve}-like random series, and studies their convergence and integrability properties.
The usual variance-based arguments cannot be applied directly, but the situation can be repaired using Kolmogorov's three series theorem, and notably the same conditions on the decay of the coefficients suffice for the heavy-tailed stable case as in the Gaussian case.

Section \ref{sec:well-posed_BIP} shows that the usual results on the Hellinger well-posedness of BIPs with respect to perturbations of the observed data and the misfit functional (negative log-likelihood) hold in the case of a heavy-tailed prior, under weaker continuity assumptions than those used to date.
Non-trivial growth of lower bounds on the misfit functional, which is typically enjoyed in applications, can and should be used to offset growth in other errors and retain well-posedness of the BIP.

\begin{figure}[t!]
	\begin{center}
		\includegraphics[width=0.6\textwidth]{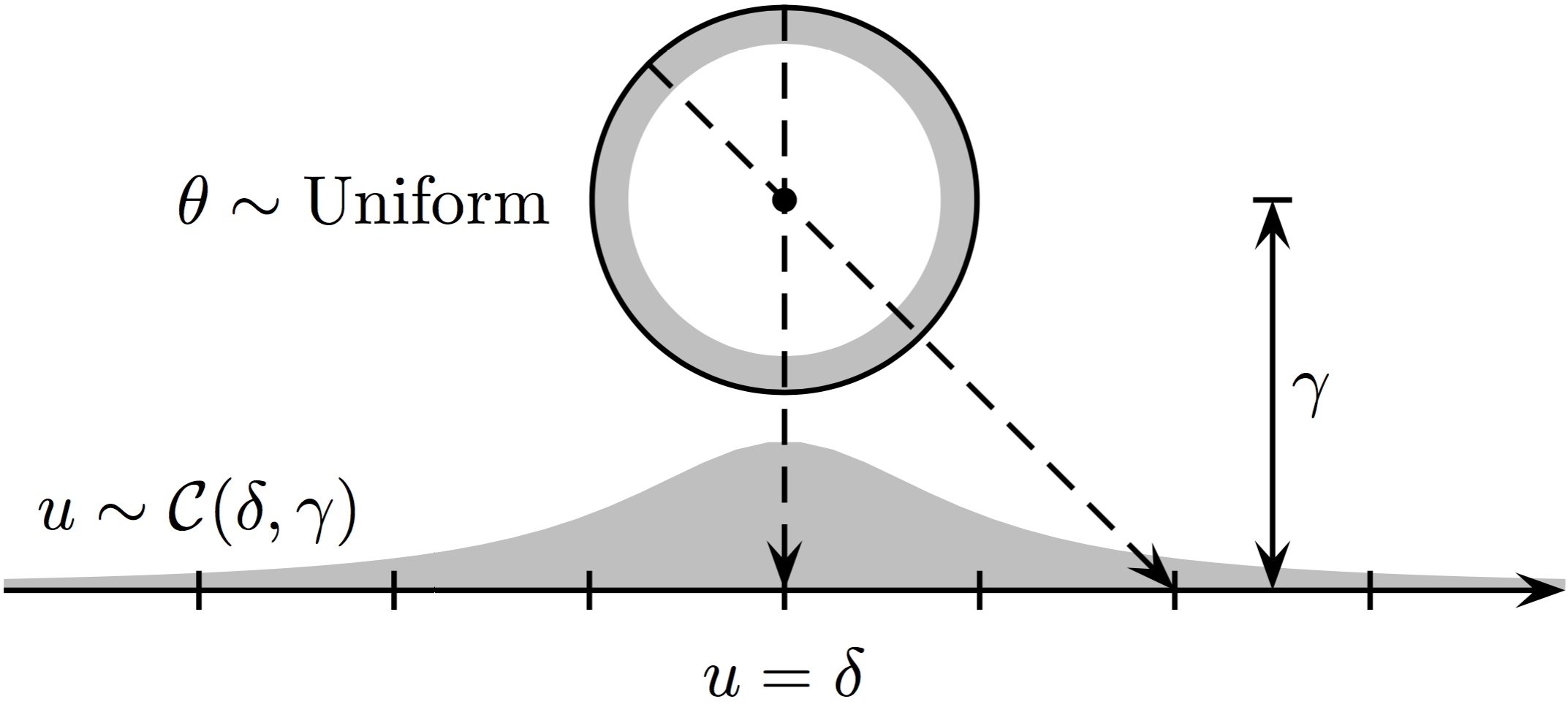}
	\end{center}
	\caption{Uniform angular measure on a circle projects radially to give Cauchy measure with width parameter $\gamma$ on any line at distance $\gamma$ from the centre of the circle.}
	\label{fig:radial}
\end{figure}

\begin{figure}[t!]
	\begin{center}
		\begin{subfigure}{0.49\linewidth}
			\centering\includegraphics[width=\linewidth]{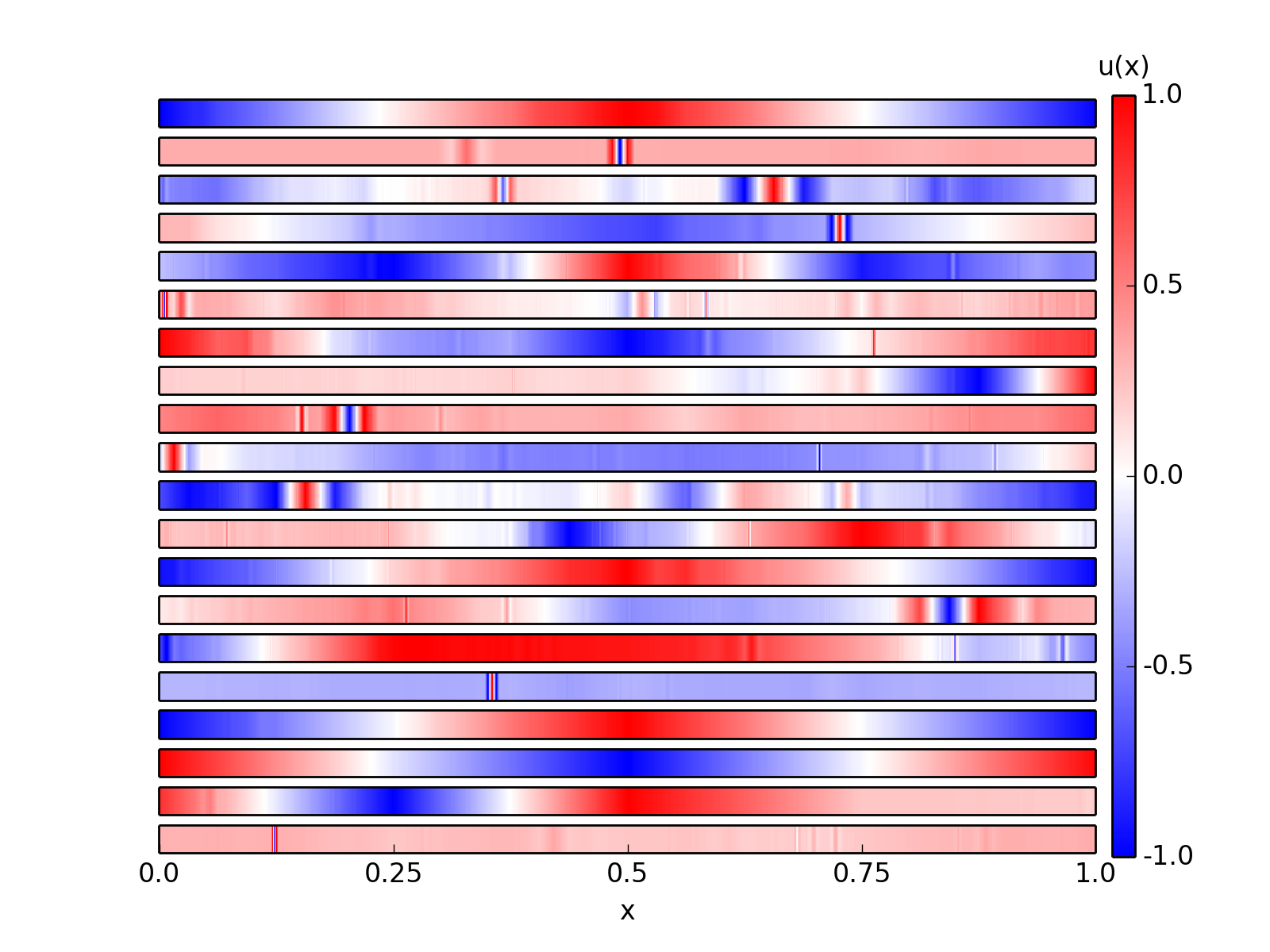}
			\caption{Linear splines, Cauchy coefficients}
		\end{subfigure}
		\begin{subfigure}{0.49\linewidth}
			\centering\includegraphics[width=\linewidth]{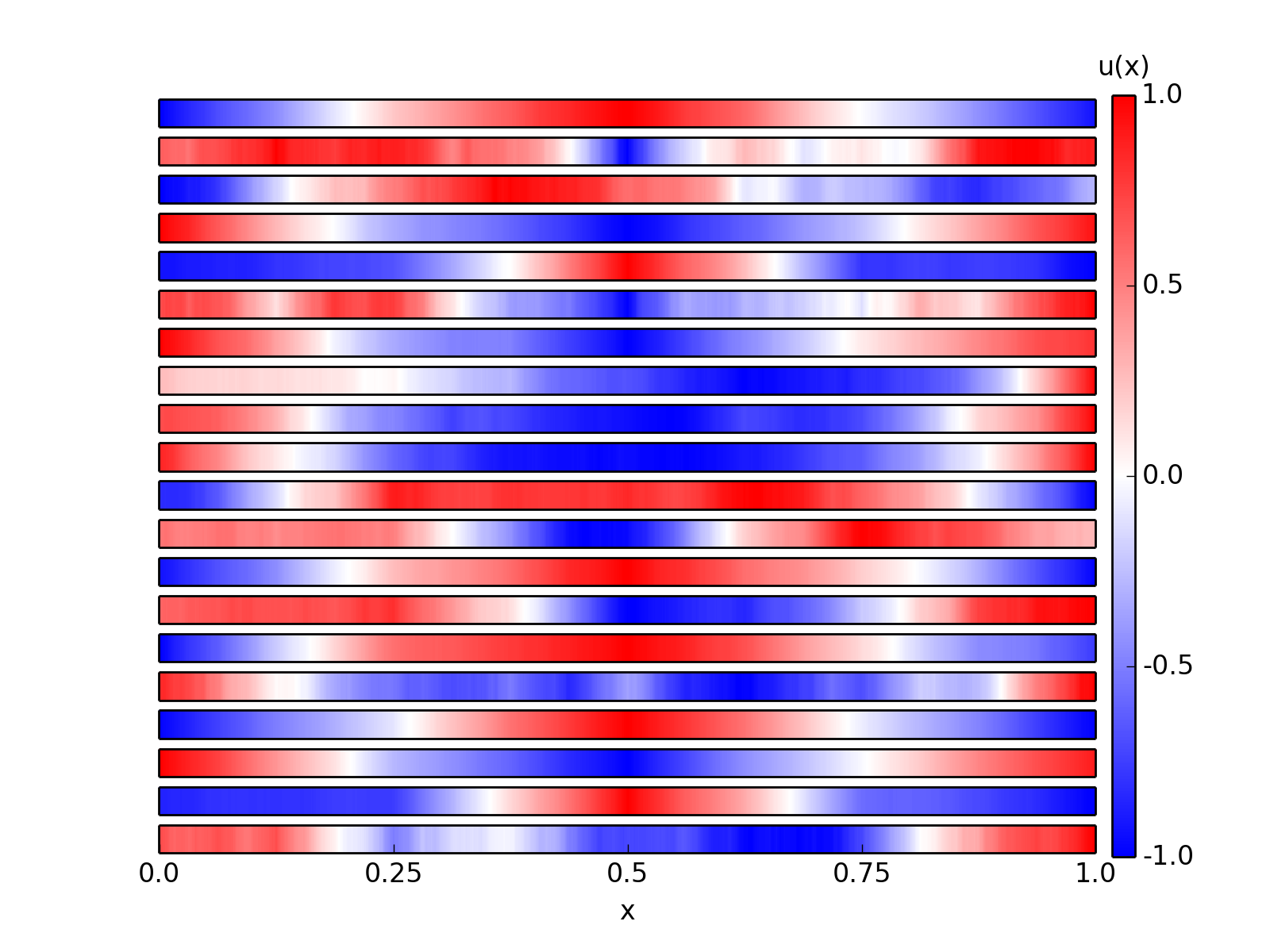}
			\caption{Linear splines, Gaussian coefficients}
		\end{subfigure}
	\end{center}
	\caption{Cauchy and Gaussian wavelet expansions in the linear spline orthonormal basis of $L^{2}([0, 1], \rd x)$.
	Each horizontal stripe shows a random function $u(x) = \sum_{j = 0}^{J} \sum_{k = 0}^{2^{j} - 1} u_{j, k} 2^{j / 2} \psi(2^{j} x - k)$, where each $u_{j , k} = (j + 1)^{-2} 2^{-j}$ times a standard Cauchy or normal draw, and $\psi$ denotes the mother wavelet.
	The plots show $20$ i.i.d.\ samples with $J = 10$.
	Theorem \ref{thm:almost_sure_convergence} ensures a.s.\ convergence in $L^{2}([0, 1])$ as $J \to \infty$.
	To enable easy comparisons between plots, the ensemble has been translated and linearly scaled to take values $u(x) \in [0, 1]$, and the same random seed is used in each case.
	Note well the large local deviations in the Cauchy case.}
	\label{fig:wavelets}
\end{figure}


\section{Background and notation}
\label{sec:notation}

\subsection{General notation}
\label{subsec:notation}

The setting for the inference problems in this paper will be a real and separable Banach or quasi-Banach space $\UU$.
Observed data will take values in another real and separable Banach or quasi-Banach space $\YY$.
Recall that in a quasi-Banach space the triangle inequality only holds in the weaker form
\[
	\| u + v \| \leq C \bigl( \| u \| + \| v \| \bigr)
\]
for some constant $C \geq 1$.
Typical examples of quasi-Banach spaces that are not Banach spaces include the $\ell^{p}$ and $L^{p}$ spaces for $0 < p < 1$.

Occasionally, we will need to make reference to an underlying probability space $(\Omega, \mathcal{F}, \mathbb{P})$ as a common domain of definition for all the $\R$-, $\UU$-, and $\YY$-valued random variables of interest.
$\one[P]$ denotes the indicator function of a measurable set or logical predicate $P$, e.g.\
\[
	\one[x \in E]
	\defeq
	\begin{cases}
		1, & \text{if $x \in E$,} \\
		0, & \text{if $x \notin E$.}
	\end{cases}
\]
A property will be said to hold almost surely if it fails only on a subset of a measurable set of measure zero, and this will be abbreviated to ``a.s.''
If $f \colon \UU \to \R$ is measurable, then $\E_{u \sim \mu}[f(u)]$ or simply $\E[f]$ denotes the expected value (Lebesgue integral) of $f$ with respect to $\mu$:
\[
	\E[f] \equiv \E_{u \sim \mu}[f(u)] \defeq \int_{\UU} f(u) \, \rd \mu(u).
\]
Equality in distribution (equality in law) for random variables $u$ and $v$ will be denoted $u \disteq v$.

The set of all Borel probability measures on $\UU$ will be denoted $\mathcal{M}_{1}(\UU)$, and $\Hellinger$ denotes the Hellinger metric on $\mathcal{M}_{1}(\UU)$, defined by
\begin{equation}
	\label{eq:Hellinger}
	\Hellinger(\mu, \nu)^{2} = \int_{\UU} \left| \sqrt{ \frac{\rd \mu}{\rd \lambda}(u) } - \sqrt{ \frac{\rd \nu}{\rd \lambda}(u) } \right|^{2} \, \rd \lambda(u) ,
\end{equation}
where $\lambda$ is any $\sigma$-finite Borel measure on $\UU$ with respect to which both $\mu$ and $\nu$ are absolutely continuous, e.g.\ $\lambda \defeq \mu + \nu$.
By Kraft's inequality \citep{Kraft:1955}, the Hellinger topology coincides with the total variation topology;
by Pinsker's inequality \citep{Pinsker:1964}, the Hellinger topology is strictly weaker than the Kullback--Leibler (relative entropy) topology;
all these topologies are strictly stronger than the topology of weak convergence of measures.
Expected values of square-integrable functions are Lipschitz continuous with respect to the Hellinger metric:
\begin{equation}
	\label{eq:Hellinger_bound}
	\bigl| \E_{\mu}[f] - \E_{\nu}[f] \bigr| \leq \sqrt{2} \sqrt{ \E_{\mu} \bigl[ |f|^{2} \bigr] + \E_{\nu} \bigl[ |f|^{2} \bigr] } \, \Hellinger(\mu, \nu)
\end{equation}
when $f \in L^{2}(\UU, \mu) \cap L^{2}(\UU, \nu)$.
In particular, $| \E_{\mu}[f] - \E_{\nu}[f] | \leq 2 \| f \|_{\infty} \Hellinger(\mu, \nu)$.

\subsection{Bayesian inverse problems}
\label{subsec:BIP}

This paper is concerned with inverse problems of the following form:
given spaces $\UU$ and $\YY$, and a known forward operator $G \colon \UU \to \YY$, recover $u \in \UU$ from a randomly corrupted observation $y \in \YY$ of $G(u)$.
A simple example is an inverse problem with additive noise, e.g.
\begin{equation}
	\label{eq:additive_noise}
	y = G(u) + \eta ,
\end{equation}
where $\eta$ is a draw from a $\YY$-valued random variable;
crucially, we assume knowledge of the probability distribution of $\eta$, but not its exact value.

Inverse problems are typically ill-posed in the sense of having no solution, or multiple solutions, or solutions that depend sensitively upon the observed data $y$.
While there is a long tradition dating back to \citet{Tikhonov:1963} and others of addressing such problems using regularisation, the Bayesian approach \citep{KaipioSomersalo:2005, Stuart:2010} is to interpret both $u$ and $y$ as random variables, and relations such as \eqref{eq:additive_noise} as defining the conditional distribution of $y$ given $u$.
First, one must posit prior beliefs about $u$ independent of $y$ in the form of a prior distribution $\mu_{0} \in \mathcal{M}_{1}(\UU)$.
Then, the Bayesian inverse problem (BIP) is to compute the posterior distribution $\mu^{y} \in \mathcal{M}_{1}(\UU)$, i.e.\ the conditional distribution of $u$ given $y$.
Naturally, one hopes to do this through an appropriate version of the Bayes formula, e.g.\ for probability densities with respect to Lebesgue measure on $\R^{n}$,
\[
	\rho^{y}(u) \equiv \rho(u|y) \propto \rho(y|u) \rho_{0}(u) ;
\]
in the case $\dim \UU = \infty$, in which there is no canonical choice of reference measure such as Lebesgue measure, this formula must be treated with some care.

As observed by \citet[Section 6.6]{Stuart:2010}, the correct statement of the Bayes formula when $\mu_{0}$ is supported on an infinite-dimensional parameter space $\UU$ is that the posterior $\mu^{y}$ has a probability density (Radon--Nikod\'ym derivative) with respect to $\mu_{0}$, and this density is proportional to the conditional probability density of $y | u$.
It is both mathematically and computationally convenient to express this relationship in exponential form.
That is, $\Phi \colon \UU \times \YY \to \R$ will denote the misfit or negative log-likelihood, meaning that, under the hypothesis that $u \in \UU$ is `correct', the probability distribution of $y | u$ is
\[
	\mathbb{P} [ y \in E | u ] = \left. \int_{E} \exp(- \Phi(u; y)) \, \rd \varrho(y) \middle/ \int_{\YY} \exp(- \Phi(u; y)) \, \rd \varrho(y) \right. ,
\]
where $\varrho$ is some $\sigma$-finite reference measure on $\YY$;
it is implicitly assumed that $y|u$ is absolutely continuous with respect to $\varrho$ for every $u \in \UU$.

In this setting, the generalised Bayes formula is
\begin{align}
	\label{eq:Bayes}
	\frac{\rd \mu^{y}}{\rd \mu_{0}} (u) & = \frac{\exp( - \Phi(u; y))}{Z(y)} ,\\
	\notag
	Z(y) & = \E_{u \sim \mu_{0}} \bigl[ \exp( - \Phi(u; y)) \bigr].
\end{align}
However, care must still be taken to check that this formula does define a probability measure $\mu^{y}$ on $\UU$;
in particular, the normalisation constant $Z(y)$ must be strictly positive and finite, and verifying this property for the stable priors $\mu_{0}$ of interest in this paper is the business of Theorem \ref{thm:posterior_defined}.

\begin{example}
	In the additive case \eqref{eq:additive_noise} with $\eta \sim \Normal(0, \Sigma)$ independently of $u$, on $\YY = \R^{n}$, with $\varrho = n$-dimensional Lebesgue measure,
	\[
		\Phi(u; y) = \tfrac{1}{2} \bigl\| \Sigma^{-1/2} ( y - G(u) ) \bigr\|_{2}^{2} .
	\]
	If $\dim \YY$ is infinite and $\eta \sim \Normal(0, \Sigma)$ is a Gaussian random variable on $\YY$ with Cameron--Martin space $\mathop{\textup{ran}}(\Sigma^{1/2})$, then this $\Phi$ is a.s.\ infinite since $y \notin \mathop{\textup{ran}}(\Sigma^{1/2})$ a.s.
	It is then necessary to `subtract off the infinite part of $\Phi$' by using the Cameron--Martin formula for translations of $\eta$ \citep[Remark 3.8]{Stuart:2010}.
\end{example}

\subsection{Stable distributions}
\label{subsec:stable}

Stable distributions have been studied extensively in the statistical and probabilistic literature.
A random variable $u$ is \defterm{stable} if, whenever $u_{1}, \dots, u_{n}$ are independent copies of $u$ and $a_{1}, \dots, a_{n} > 0$, $\sum_{i = 1}^{n} a_{i} u_{i} \disteq c u + d$ for some $c > 0$ and $d \in \R$.
The random variable is \defterm{strictly stable} if this holds with $d = 0$ for all choices of the $a_{i}$.
This relation can be made more quantitatively precise:
$u$ is \defterm{stable of order $\alpha \in (0, 2]$} if $\sum_{i = 1}^{n} u_{i} \disteq n^{1 / \alpha} u + d$.
Equivalently, in terms of the law $\mu$ of $u$ and the rescaling $\mu_{n}(E) \defeq \mu(n^{1/\alpha} E)$,
\[
	\mu = \underbrace{( \mu_{n} \star \dots \star \mu_{n} )}_{\text{$n$-fold convolution}} ( E + d ) 
	\quad
	\text{for all Borel-measurable $E$.}
\]
Stability is a particularly appealing property if the aim is to construct prior measures for BIPs that are `physically consistent' in the sense of remaining in the same model class regardless of discretisation or coordinate choices, at least when the `physical quantity' obeys an additive law.\footnote{\label{footnote:harmonic}There are situations where such an additive decomposition is not appropriate:
e.g., the average homogenised permeability for Darcy flow given the permeabilities of smaller grid cells is obtained as the harmonic rather than arithmetic mean.}

\begin{example}
	Suppose that the aim is to model (and later infer, in a Bayesian fashion) the distribution of electrical charge in some domain $\Omega \subseteq \R^{3}$.
	For computational purposes, $\Omega$ is approximated by a triangulation $\mathcal{T}$.
	Consider two elements $T_{1}, T_{2} \in \mathcal{T}$.
	If $\mathop{\textup{charge}}(T_{i})$ is stably distributed, then so too is
	\[
		\mathop{\textup{charge}}(T_{1} \cup T_{2}) = \mathop{\textup{charge}}(T_{1}) + \mathop{\textup{charge}}(T_{2}) .
	\]
	The charge density $\mathop{\textup{charge}}(T_{i}) / \mathop{\textup{volume}}(T_{i})$ behaves similarly.
	Thus, we remain in the same stable model class if we coarsen or refine the mesh $\mathcal{T}$;
	this would not be true for an unstable random model of the charge, and this would complicate computational modelling in an undesirable fashion.
\end{example}

Stable and strictly stable distributions on Banach spaces $\UU$, and indeed on locally convex topological vector spaces, can be defined in the same way as in the univariate case, by reference to sums of independent copies or convolutions of their laws \citep[Section 4.2]{Bogachev:2010}.
It can be shown that $\mu \in \mathcal{M}_{1}(\UU)$ is stable of order $\alpha$ precisely when all of its finite dimensional projections are stable of order $\alpha$, and if all one-dimensional projections of $\mu$ are strictly stable of order $\alpha$, then so is $\mu$.
These facts motivate further examination of stable distributions on $\R$.

Real-valued stable random variables are completely classified by four parameters, and of the many possible parametrisations, this article will follow ``Parametrisation 0'' of \citet{Nolan:2015}:
a random variable $u$ will be said to be \defterm{stably distributed} with index of stability $\alpha \in (0, 2]$, skewness $\beta \in [-1, 1]$, scale parameter $\gamma \geq 0$, and location parameter $\delta \in \R$, denoted $u \sim \Stable(\alpha, \beta, \gamma, \delta; 0)$, if the characteristic function (inverse Fourier transform) of $u$ satisfies
\[
	\E \bigl[ \exp( i t u ) \bigr] = 
	\begin{cases}
		\exp \bigl( i \delta t  - | \gamma t |^{\alpha} [ 1 + i \beta ( \tan \tfrac{\pi \alpha}{2} ) (\sign t) ( | \gamma t |^{1 - \alpha} - 1 ) ] \bigr) & \text{if $\alpha \neq 1$,} \\
		\exp \bigl(i \delta t  - | \gamma t | [ 1 + i \beta \tfrac{2}{\pi} (\sign t) \log \gamma | t | ] \bigr) & \text{if $\alpha = 1$.}
	\end{cases}
\]
(The convention here is that $0 \log 0 \defeq \lim_{s \searrow 0} s \log s = 0$.)
If $\gamma = 1$ and $\delta = 0$, then $u$ is said to be \defterm{standardised} and we write $u \sim \Stable(\alpha, \beta; 0)$.
When $\beta = \delta = 0$, $u$ is said to be \defterm{symmetric} and we obtain another common characterisation of (symmetric) $\alpha$-stable random variables:
those random variables with characteristic function $\exp( - | \gamma t |^{\alpha} )$.

Although a stable random variable $u$ can easily be shown to have a smooth Lebesgue density, exact formulae for this density are not available except in special cases.
In particular, the normal distribution with mean $m$ and standard deviation $\sigma$ is $\Stable(2, 0, \sigma / \sqrt{2}, m; 0)$, and $\Cauchy(\delta, \gamma) = \Stable(1, 0, \gamma, \delta; 0)$.
The stability properties of $\Stable(\alpha, \beta, \gamma, \delta; 0)$ distributions are summarised by the following result:

\begin{proposition}[{\citealp[Proposition 1.16]{Nolan:2015}}]
	\label{prop:Nolan_Proposition_1.16}
	If $u \sim \Stable(\alpha, \beta, \gamma, \delta; 0)$, then, for $a \neq 0$ and $b \in \R$,
	\[
		a u + b \sim \Stable(\alpha, (\sign a) \beta, | a | \gamma, a \delta + b; 0).
	\]
	Also, if $u_{1} \sim \Stable(\alpha, \beta_{1}, \gamma_{1}, \delta_{1}; 0)$ and $u_{2} \sim \Stable(\alpha, \beta_{2}, \gamma_{2}, \delta_{2}; 0)$ are independent, then $u_{1} + u_{2} \sim \Stable(\alpha, \beta, \gamma, \delta; 0)$ with
	\begin{align*}
		\beta & \defeq \frac{\beta_{1} \gamma_{1}^{\alpha} + \beta_{2} \gamma_{2}^{\alpha}}{\gamma_{1}^{\alpha} + \gamma_{2}^{\alpha}} , \\
		\gamma^{\alpha} & \defeq \gamma_{1}^{\alpha} + \gamma_{2}^{\alpha} , \\
		\delta & \defeq
		\begin{cases}
			\delta_{1} + \delta_{2} + ( \tan \tfrac{\pi \alpha}{2} ) ( \beta \gamma - \beta_{1} \gamma_{1} - \beta_{2} \gamma_{2} ) , & \text{if $\alpha \neq 1$,} \\
			\delta_{1} + \delta_{2} + \tfrac{2}{\pi} (\beta \gamma \log \gamma - \beta_{1} \gamma_{1} \log \gamma_{1} - \beta_{2} \gamma_{2} \log \gamma_{2}) , & \text{if $\alpha = 1$.}
		\end{cases}
	\end{align*}
\end{proposition}

The stable distributions with $\alpha = 2$ are exactly the Gaussian measures (the skewness parameter $\beta$ has no effect, and is conventionally set to $0$):
by Fernique's theorem, Gaussian measures are exponentially integrable, and in particular have polynomial moments of all orders.
Conversely, for $\alpha \in (0, 2)$, the stable distributions are all heavy-tailed:
when $u \sim \Stable(\alpha, \beta, \gamma, 0; 0)$ with $0 < \alpha < 2$,
\begin{equation}
	\label{eq:stable_flom}
	\E \bigl[ | u |^{p} \bigr] = 
	\begin{cases}
		C_{\alpha, \beta} \gamma^{\alpha} < \infty, & \text{for $0 < p < \alpha$,} \\
		\infty, & \text{for $p \geq \alpha$.}
	\end{cases}
\end{equation}
The asymptotic behaviour of the cumulative distribution and probability density functions of $u \sim \Stable(\alpha, \beta, \gamma, \delta; 0)$, with $0 < \alpha < 2$ and $- 1 < \beta \leq 1$, is that of a power law \citep[Theorem 1.12]{Nolan:2015}:
\begin{align}
	\label{eq:stable_tail_cdf}
	\mathbb{P}[u > x] & \sim c_{\alpha} \gamma^{\alpha} ( 1 + \beta ) x^{- \alpha} & & \text{as $x \to \infty$,} \\
	\label{eq:stable_tail_pdf}
	\rho_{u}(x) & \sim c_{\alpha} \alpha \gamma^{\alpha} ( 1 + \beta ) x^{- (\alpha + 1)} & & \text{as $x \to \infty$.}
\end{align}
Similar expressions hold for the behaviour as $x \to - \infty$.
Henceforth, to avoid some technical complications, we assume that $-1 < \beta < 1$, so that $u \sim \Stable(\alpha, \beta, \gamma, \delta; 0)$ is supported on the whole of $\R$.

One further theoretical argument in favour of modelling using stable random variables, particularly from a limiting mesh refinement point of view, is that the stable distributions are precisely the central limits of independent and identically distributed random variables:

\begin{theorem}[Generalised central limit theorem:  {\citealp[Theorem 1.20]{Nolan:2015}}]
	\label{thm:general_CLT}
	A non-degenerate random variable $u$ is $\Stable(\alpha, \beta, \gamma, \delta; 0)$ if and only if there is a sequence of i.i.d.\ random variables $x_{1}, x_{2}, \dots$ and constants $a_{n} > 0$, $b_{n} \in \R$ such that $b_{n} + a_{n} \sum_{i = 1}^{n} x_{i}$ converges in distribution to $u$ as $n \to \infty$.
\end{theorem}

The literature contains many further application-specific arguments for or against the use of stable distributions in optimisation and inference.
\citet{OHagan:1988} gives a general perspective on modelling with heavy-tailed distributions.
\citet{ShaoNikias:1993} treat applications to physics, biology, and electrical engineering, particularly for the modelling of signals and noises with occasional sharp spikes or bursts, as in Figure \ref{fig:wavelets}.
\citet{TsakalidesReveliotisNikias:2000} and \citet{AchimTsakalidesBezerianos:2003} treat applications to communications and image processing, while \citep{Tsionas:1999} discusses applications to economics.
\citet{Hansen:2006} discuss applications to optimisation.

Of particular relevance to this article is the recent work of \citet{MarkkanenRoininenHuttunenLasanen:2016}, which proposes the use of heavy-tailed priors for edge-preserving (i.e.\ non-smoothing) Bayesian inversion in X-ray tomography;
in essence, a Cauchy prior is placed on the gradient of the image to be reconstructed, thereby allowing for jump discontinuities in the image.
One objective of this article is to provide a well-posedness theory in the style of \citet{Stuart:2010} to underwrite the numerical investigations of \citet{MarkkanenRoininenHuttunenLasanen:2016}.

\section{Karhunen--{Lo\`eve} expansions for stable distributions on quasi-Banach spaces}
\label{sec:sampling}

Now consider the problem of constructing and sampling heavy-tailed stable probability measures on a real quasi-Banach space $\UU$, for example a vector space of summable sequences or a Sobolev space of fields of specified smoothness.
Supposing that one already has access to a generator of real-valued stable random variables \citep{ChambersMallowsStuck:1976}, it is natural to try to realise a $\UU$-valued stable random variable via an infinite random series of the form
\begin{equation}
	\label{eq:random_series}
	u \defeq \sum_{n \in \N} u_{n} \psi_{n},
\end{equation}
where the $\psi_{n}$ are a basis for $\UU$ and the $u_{n}$ are $\R$-valued stable random variables;
this is the strategy used to generate the examples shown in Figure \ref{fig:wavelets}.
The natural question is, when does \eqref{eq:random_series} define a bona fide $\UU$-valued random variable?

The Gaussian case is a useful reference point.
Suppose that $C$ is a positive-semi-definite and self-adjoint operator on a Hilbert space $\UU$ with an eigensystem $(\lambda_{n}, \psi_{n})_{n \in \N}$, and that $(\lambda_{n})_{n \in \N} \in \ell^{1}$, i.e.\ $C$ is a trace-class operator.
Then the series \eqref{eq:random_series} with $u_{n} \sim \Normal(0, \lambda_{n})$ --- i.e.\ with $u_{n} = \lambda_{n}^{1/2} \hat{u}_{n}$ with $\hat{u}_{n} \sim \Normal(0, 1)$ --- converges a.s., and is a draw from the Gaussian measure $\Normal(0, C)$ on $\UU$ with covariance operator $C$.
Similar expansions with different powers of $\lambda_{n}$ and $\hat{u}_{n}$ having density proportional to $\exp( - | \hat{u}_{n} |^{p} )$ are used to define draws from Besov measures \citep{DashtiHarrisStuart:2012}.

However, the focus here is on $u_{n}$ with heavy tails, so the usual variance-based arguments that are used to prove a.s.\ convergence of the series \eqref{eq:random_series} will not be applicable.
However, Theorem \ref{thm:almost_sure_convergence} below shows that the series \eqref{eq:random_series} indeed converges almost surely in $\UU$ under the assumption that the scale parameters $\gamma_{n}$ of the stable random coefficients $u_{n}$ are $\alpha$-summable, modulo a logarithmic correction term in the case $\alpha = 1$.
The proof of this rests on the following result, which is a synthesis of two classical results from probability theory, and gives a necessary and sufficient condition for the convergence of random series:

\begin{theorem}[Kolmogorov's zero-one law and three series theorem]
	\label{thm:Kolmogorov_3}
	Let $(x_{n})_{n \in \N}$ be a sequence of independent $\R$-valued random variables.
	Then the series $\sum_{n \in \N} x_{n}$ either converges a.s.\ or diverges a.s, and a.s.\ convergence holds if and only if, for some $A > 0$, the following series are all finite:
	\[
		\sum_{n \in \N} \mathbb{P} \bigl[ | x_{n} | > A \bigr], \quad
		\sum_{n \in \N} \mathbb{E} \bigl[ x_{n} \one [ | x_{n} | \leq A ] \bigr] , \quad
		\text{and} \quad
		\sum_{n \in \N} \mathbb{E} \bigl[ x_{n}^{2} \one [ | x_{n} | \leq A ] \bigr] .
	\]
\end{theorem}

\begin{definition}
	\label{defn:U-stable}
	Let $\UU$ be a real quasi-Banach space with countable, unconditional, normalised, Schauder basis $(\psi_{n})_{n \in \N}$.
	Let $\alpha \in (0, 2]$, $\bbeta = (\beta_{n})_{n \in \N} \subset (-1, 1)$, $\ggamma = (\gamma_{n})_{n \in \N} \subset \R_{+}$, and $\ddelta = (\delta_{n})_{n \in \N} \subset \R$.
	Let $u_{n} \sim \Stable(\alpha, \beta_{n}, \gamma_{n}, \delta_{n}; 0)$ be independent for each $n \in \N$.
	Then we shall say that $u \defeq \sum_{n \in \N} u_{n} \psi_{n}$ is a \defterm{stable $\UU$-valued random variable} and write $u \sim \Stable(\alpha, \bbeta, \ggamma, \ddelta; 0)$.
\end{definition}

Theorem \ref{thm:almost_sure_convergence} will justify the terminology of Definition \ref{defn:U-stable} by showing that, under suitable summability conditions on $\ggamma$ and $\ddelta$, $u \sim \Stable(\alpha, \bbeta, \ggamma, \ddelta; 0)$ is indeed a well-defined $\UU$-valued random variable.
First, it is necessary to make an assumption on the geometry of the basis $(\psi_{n})_{n \in \N}$.

\begin{assumption}
	\label{ass:q-frame_upper}
	The basis $(\psi_{n})_{n \in \N}$ and $q > 0$ are such that the synthesis operator $S_{\psi} \colon \underline{v} \defeq (v_{n})_{n \in \N} \mapsto \sum_{n \in \N} v_{n} \psi_{n}$ is a continuous embedding of the sequence space $\ell^{q}$ of coefficients into $\UU$, i.e.
	\begin{equation}
		\label{eq:q-frame_upper}
		\Biggl\| \sum_{n \in \N} v_{n} \psi_{n} \Biggr\|_{\UU} \leq C \| \underline{v} \|_{\ell^{q}}.
	\end{equation}
\end{assumption}

When $\UU$ is a Banach space, Assumption \ref{ass:q-frame_upper} holds with $q = 1$ for any choice of basis $(\psi_{n})_{n \in \N}$, since it is just the triangle inequality for an unconditionally convergent series in $\UU$.
Since $0 < p \leq q \leq \infty \implies \| \quark \|_{\ell^{q}} \leq \| \quark \|_{\ell^{p}}$, whenever \eqref{eq:q-frame_upper} holds for $q$ it also holds with $q$ replaced by any $p \in (0, q]$.
If inequality \eqref{eq:q-frame_upper} can be reversed, possibly with a different constant, then the basis $(\psi_{n})_{n \in \N}$ is known as a \defterm{$q$-frame} for $\UU$ \citep{ChristensenStoeva:2003}.
The case $q = 2$ is the well-known notion of a \defterm{Riesz basis}.

\begin{theorem}[Well-definedness of $\UU$-valued stable random variables]
	\label{thm:almost_sure_convergence}
	Let $u \sim \Stable(\alpha, \bbeta, \ggamma, \ddelta; 0)$ with $\alpha \in (0, 2)$, $\bbeta \subset (-1, 1)$, $\ggamma \in \ell^{\alpha}$, $\ddelta \in \ell^{q}$ and, in addition,
	\begin{align}
		\label{eq:Orlicz}
		[ \ggamma ]_{\ell^{\alpha} \log \ell} & \defeq \sum_{n \in \N} \bigl| \gamma_{n}^{\alpha} \log | \gamma_{n} | \bigr| < \infty, & & \text{if $\alpha = q$ or $2 q$.}
	\end{align}
	Then $u \in \UU$ a.s.
\end{theorem}

\begin{proof}
	For each $n \in \N$, let $\hat{u}_{n} \sim \Stable(\alpha, \beta_{n}; 0)$, so that $u_{n} \disteq \delta_{n} + \gamma_{n} \hat{u}_{n}$.
	For $M, N \in \N$ with $N > M$,
	\begin{align*}
		& \Biggl\| \sum_{n = 1}^{N} u_{n} \psi_{n} - \sum_{n = 1}^{M} u_{n} \psi_{n} \Biggr\|_{\UU} 
		= \Biggl\| \sum_{n = M + 1}^{N} (\delta_{n} + \gamma_{n} \hat{u}_{n}) \psi_{n} \Biggr\|_{\UU} \\
		& \quad \leq C \Biggl\| \sum_{n = M + 1}^{N} \delta_{n} \psi_{n} \Biggr\|_{\UU} + C \Biggl\| \sum_{n = M + 1}^{N} \gamma_{n} \hat{u}_{n} \psi_{n} \Biggr\|_{\UU} & & \text{since $\| \quark \|_{\UU}$ is a quasi-norm} \\
		& \quad \leq C \sum_{n = M + 1}^{N} | \delta_{n} |^{q} + C \sum_{n = M + 1}^{N} | \gamma_{n} \hat{u}_{n} |^{q} & & \text{by Assumption \ref{ass:q-frame_upper}.}
	\end{align*}
	If $\ddelta \in \ell^{q}$, then the dominated convergence theorem implies that (deterministic) first sum on the right-hand side converges to $0$ as $N, M \to \infty$.
	Therefore, it remains only to show that the assumptions on $\ggamma$ are sufficient to ensure that the (random) second sum on the right-hand side converges a.s.\ to $0$ as $N, M \to \infty$;
	it will then follow that the partial sums of $u$ are a.s.\ Cauchy in the quasi-Banach space $\UU$, and hence a.s.\ convergent to a well-defined limit in $\UU$.
	
	To that end, it will be shown that $\sum_{n \in \N} | \gamma_{n} \hat{u}_{n} |^{q}$ converges a.s.\ in $\R$.
	Let $A > 0$ be large enough that the asymptotic properties \eqref{eq:stable_tail_cdf} and \eqref{eq:stable_tail_pdf} hold true for $| x | > A$.
	Then, by \eqref{eq:stable_tail_cdf},
	\[
		\mathbb{P} \bigl[ | \gamma_{n} \hat{u}_{n} |^{q} > A \bigr] \sim C \left| \frac{\gamma_{n}}{A^{1/q}} \right|^{\alpha},
	\]
	where $C$ depends only on $\alpha$ and $\beta_{n}$.
	Since $\ggamma \in \ell^{\alpha}$, the series $\sum_{n \in \N} \mathbb{P} \bigl[ | \gamma_{n} \hat{u}_{n} |^{q} > A \bigr]$ is convergent.
	By \eqref{eq:stable_tail_pdf}, for $p = 1, 2$, the truncated $p^{\text{th}}$ moments of $| \gamma_{n} \hat{u}_{n} |^{q}$ satisfy
	\begin{align*}
		\E \bigl[ | \gamma_{n} \hat{u}_{n} |^{p q} \one [ | \gamma_{n} \hat{u}_{n} |^{q} < A ] \bigr] 
		& = | \gamma_{n} |^{p q} \int_{- A^{1/q} / \gamma_{n}}^{A^{1/q} / \gamma_{n}} | s |^{p q} \rho_{\hat{u}_{n}} (s) \, \rd s \\
		& \leq 
		\begin{cases}
			C | \gamma_{n}^{\alpha} \log \gamma_{n} | & \text{if $p q = \alpha$,} \\
			C | \gamma_{n} |^{\alpha}, & \text{otherwise,}
		\end{cases}
	\end{align*}
	where $C$ depends on $\alpha$, $\beta_{n}$, $p$, and $A$ but is independent of $\gamma_{n}$.
	The assumptions on $\ggamma$ ensure that these truncated moments are both summable over all $n \in \N$ for $p = 1$ and $p = 2$.
	Therefore, Theorem \ref{thm:Kolmogorov_3} implies that $\sum_{n \in \N} | \gamma_{n} \hat{u}_{n} |^{q}$ converges a.s.\ in $\R$, and so \eqref{eq:random_series} converges a.s.\ in $\UU$.
\end{proof}

\begin{example}
	\label{eg:almost_sure_convergence_Cauchy}
	Suppose that $\UU$ is a Banach space (so we may take $q = 1$, but perhaps no greater).
	If the coefficients $u_{n}$ in \eqref{eq:random_series} are independent Cauchy random variables, $u_{n} \disteq \gamma_{n} \hat{u}_{n} \sim \Cauchy(0, \gamma_{n}) = \Stable(1, 0, \gamma_{n}, 0; 0)$, then the truncated moments are
	\begin{align}
		\label{eq:Cauchy_moment_0}
		\mathbb{P} \bigl[ | \gamma_{n} \hat{u}_{n} | \geq A \bigr] & = 1 - \frac{2}{\pi} \arctan \frac{A}{\gamma_{n}} , \\
		\label{eq:Cauchy_moment_1}
		\E \bigl[ | \gamma_{n} \hat{u}_{n} | \one [ | \gamma_{n} \hat{u}_{n} | < A ] \bigr] & = \frac{\gamma_{n}}{\pi} \log \left( 1 + \frac{A^{2}}{\gamma_{n}^{2}} \right) , \\
		\label{eq:Cauchy_moment_2}
		\E \bigl[ | \gamma_{n} \hat{u}_{n} |^{2} \one [ | \gamma_{n} \hat{u}_{n} | < A ] \bigr] & = \frac{2 A \gamma_{n}}{\pi} + \frac{2 \gamma_{n}^{2}}{\pi} \arctan \frac{A}{\gamma_{n}} .
	\end{align}
	Consistent with Theorem \ref{thm:almost_sure_convergence}, the corresponding three series all converge if $\| \ggamma \|_{\ell^{1}}$ and $[ \ggamma ]_{\ell \log \ell}$ are finite, and in particular if $\gamma_{n} = \textup{O}(n^{- r})$ for some $r > 1$.
	When this convergence holds, the random series \eqref{eq:random_series} converges a.s.\ in $\UU$, and thereby defines a $\UU$-valued Cauchy random variable.
	
	This example also shows that Theorem \ref{thm:almost_sure_convergence} is sharp:
	for $\UU = \ell^{1}$ with its standard Euclidean basis $( \psi_{n} )_{n \in \N}$, with $u_{n} \sim \Cauchy(0, \gamma_{n})$, 
	\[
		\left\| \sum_{n = 1}^{N} u_{n} \psi_{n} - \sum_{n = 1}^{M} u_{n} \psi_{n} \right\|_{\ell^{1}} = \sum_{n = M + 1}^{N} | u_{n} | ,
	\]
	so the partial sums of $u$ are a.s.\ Cauchy in $\ell^{1}$ if and only if the real random series $\sum_{n \in \N} | u_{n} |$ is a.s.\ convergent, and the `if and only if' part of Kolmogorov's three-series theorem and exact values \eqref{eq:Cauchy_moment_0}--\eqref{eq:Cauchy_moment_2} for the truncated moments together imply that this holds exactly when $\| \ggamma \|_{\ell^{1}}$ and $[ \ggamma ]_{\ell \log \ell}$ are finite.
	
	Condition \eqref{eq:Orlicz}, requiring in this case that the Orlicz-type quantity $[ \ggamma ]_{\ell \log \ell}$ be finite, cannot generally be weakened to just requiring that $\ggamma \in \ell^{1}$.
	For example, for $\gamma_{n} \defeq n^{-1} ( \log n )^{-2}$, the integral test reveals that $\sum_{n \geq 2} | \gamma_{n} | < \infty$ but $\sum_{n \geq 2} | \gamma_{n} \log \gamma_{n} | = \infty$;
	in this situation, summability of the truncated first absolute moments of the coefficients $\gamma_{n} \hat{u}_{n}$ is no longer assured.
	However, for polynomial $\ggamma$, the $\ell^{1}$ and $\ell \log \ell$ criteria \emph{do} coincide:
	for $\gamma_{n} = C n^{- r}$, $\| \ggamma \|_{\ell^{1}}$ is finite once $r > 1$, and then $[ \ggamma ]_{\ell \log \ell}$ is also finite.

	It is worth noting in passing that, like Gaussians, infinite-dimensional Cauchy distributions of this type satisfy a Cameron--Martin-type theorem.
	It follows from \citet[Theorem 5.2.1 and Example 5.2.3]{Bogachev:2010} that the law of $u$ with $u_{n} \sim \Cauchy(0, \gamma_{n})$ is mutually equivalent with the law of the shifted random variable $v$ with $v_{n} \sim \Cauchy(h_{n}, \gamma_{n})$ precisely when $(h_{n} / \gamma_{n})_{n \in \N} \in \ell^{2}$.
	This Hilbert shift quasi-invariance space also coincides with the domain of Fomin differentiability for the law of $u$.
\end{example}

\begin{remark}
	\label{rmk:our_stables_are_stable}
	An immediate consequence of the stability of each of the coefficients $u_{n}$ in the basis $\{ \psi_{n} \}_{n \in \N}$ is that $\UU$-valued random variables in the sense of Definition \ref{defn:U-stable} are stable in the general sense of e.g.\ \citet[Section 4.2]{Bogachev:2010}.
\end{remark}

\begin{remark}[Values in Hilbert scales]
	\label{rmk:Hilbert_scales}
	Suppose that $\UU$ is a Hilbert space and $(\psi_{n})_{n \in \N}$ is an orthonormal basis or normalised Riesz basis ($2$-frame) of $\UU$.
	Theorem \ref{thm:almost_sure_convergence} shows that $u \sim \Stable(\alpha, \bbeta, \ggamma, \ddelta; 0)$ takes values in $\UU$ a.s.\ when $\ggamma \in \ell^{\alpha}$ and $\ddelta \in \ell^{2}$. 
	If, say, $\UU = L^{2}(D)$ for some domain $D \subseteq \R^{d}$, then this Hilbert setting offers an easy way to have $u$ a.s.\ take values that are fields of specified smoothness by the well-established technique of a Hilbert scale \citep{Bonic:1967}.
	For a positive-definite bounded linear operator $C$ on $\UU$, the scaled space $\UU^{s}$ is defined to be the completion of $\UU$ with respect to the inner product $\langle u, v \rangle_{\UU^{s}} \defeq \langle C^{-s} u, C^{-s} v \rangle_{\UU}$.
	A standard example is that $C = (- \Delta)^{-1/2}$, which generates the scale of Sobolev spaces on $D$.
	If the basis $(\psi_{n})_{n \in \N}$ is taken to be the eigenbasis of $C$ with eigenvalues $(\lambda_{n})_{n \in \N}$ in decreasing order and tending to $0$, then
	\[
		\UU^{s} = \Biggl\{ \sum_{n \in \N} v_{n} \psi_{n} \Bigg| \sum_{n \in \N} \lambda_{n}^{- 2 s} v_{n}^{2} < \infty \Biggr\}
	\]
	and $u \in \UU^{s}$ a.s.\ when $(\gamma_{n} / \lambda_{n}^{s})_{n \in \N} \in \ell^{\alpha}$ and $(\delta_{n} / \lambda_{n}^{s})_{n \in \N} \in \ell^{2}$.
\end{remark}


The final objective of this section is to show that $u \sim \Stable(\alpha, \bbeta, \ggamma, \ddelta; 0)$ has finite fractional lower-order moments $\E \bigl[ \| u \|_{\UU}^{q} \bigr]$ for $0 < p < \alpha$, as in the real-valued case.

\begin{theorem}[$p^{\text{th}}$-mean convergence and fractional lower-order moments]
	\label{thm:Lp_convergence}
	Let $u \sim \Stable(\alpha, \bbeta, \ggamma, \ddelta; 0)$ satisfy the assumptions of Theorem \ref{thm:almost_sure_convergence}, and suppose that $(\psi_{n})_{n \in \N}$ satisfies \eqref{eq:q-frame_upper} for some $q > 0$.
	Let $0 < p \leq q$ and $p < \alpha$.
	Then $\sum_{n = 1}^{N} u_{n} \psi_{n} \to u$ in $L^{p}(\Omega, \mathbb{P}; \UU)$ as $N \to \infty$ and, in particular,
	\begin{equation}
		\label{eq:Lp_convergence_flom}
		\| u \|_{L^{p}(\Omega, \mathbb{P}; \UU)}^{p} \equiv \E \bigl[ \| u \|_{\UU}^{p} \bigr] \leq C \| \ggamma \|_{\ell^{\alpha}} + C \| \ddelta \|_{\ell^{q}} < \infty.
	\end{equation}
\end{theorem}

\begin{proof}
	To save space, $\| u \|_{L^{p}} \defeq \bigl( \E \bigl[ \| u \|_{\UU}^{p} \bigr] \bigr)^{1/p}$ denotes the quasinorm in $L^{p}(\Omega, \mathbb{P}; \UU)$.
	Let $M, N \in \N$ with $N > M$.
	Then
	\begin{align*}
		\Biggl\| \sum_{n = 1}^{N} u_{n} \psi_{n} - \sum_{n = 1}^{M} u_{n} \psi_{n} \Biggr\|_{L^{p}}
		& \leq C \Biggl\| \sum_{n = M + 1}^{N} \delta_{n} \psi_{n} \Biggr\|_{L^{p}} + C \Biggl\| \sum_{n = 1}^{N} \gamma_{n} \hat{u}_{n} \psi_{n} \Biggr\|_{L^{p}} \\
		& = \Biggl\| \sum_{n = M + 1}^{N} \delta_{n} \psi_{n} \Biggr\|_{\UU} + \Biggl( \E \Biggl[ \Biggl\| \sum_{n = M + 1}^{N} \gamma_{n} \hat{u}_{n} \psi_{n} \Biggr\|_{\UU}^{p} \Biggr] \Biggr)^{1/p} ,
	\end{align*}
	where the inequality follows from the generalised triangle inequality for the quasinorm in $L^{p}(\Omega, \mathbb{P}; \UU)$ and the equality follows from $\ddelta$ being a deterministic sequence.
	By Assumption \ref{ass:q-frame_upper}, $\bigl\| \sum_{n = M + 1}^{N} \delta_{n} \psi_{n} \bigr\|_{\UU} \leq C \| (\delta_{n})_{n = M + 1}^{N} \|_{\ell^{q}}$;
	since $\ddelta \in \ell^{q}$, the dominated convergence theorem implies that the first term on the right-hand side of the previous display tends to zero as $M, N \to \infty$, so now we consider the second, random term.

	Since the basis $(\psi_{n})_{n \in \N}$ satisfies \eqref{eq:q-frame_upper} for $q$, \eqref{eq:q-frame_upper} also holds with $q$ replaced by $p$, and so
	\begin{align*}
		\E \Biggl[ \Biggl\| \sum_{n = M + 1}^{N} \gamma_{n} \hat{u}_{n} \psi_{n} \Biggr\|_{\UU}^{p} \Biggr] 
		& \leq C \E \Biggl[ \sum_{n = M + 1}^{N} | \gamma_{n} \hat{u}_{n} |^{p} \Biggr] & & \text{by \eqref{eq:q-frame_upper} with $p$ in place of $q$} \\
		& \leq C \sum_{n = M + 1}^{N} | \gamma_{n} |^{\alpha} & & \text{by \eqref{eq:stable_flom}.}
	\end{align*}
	Since $\ggamma \in \ell^{\alpha}$, the dominated convergence theorem implies that the right-hand side tends to zero as $M, N \to \infty$.
	Thus, the partial sums of the series $\sum_{n \in \N} u_{n} \psi_{n}$ are Cauchy in the quasi-Banach space $L^{p}(\Omega, \mathbb{P}; \UU)$, which implies that they converge to $u \in L^{p}(\Omega, \mathbb{P}; \UU)$.
	
	The estimate \eqref{eq:Lp_convergence_flom} follows from the above and Fatou's lemma.
\end{proof}

The next section uses Theorem \ref{thm:Lp_convergence} in the form that, when $\mu_{0}$ is the law of an $\alpha$-stable $u$, $\exp (p \log \| \quark \|_{\UU}) \in L^{1}(\UU, \mu_{0})$ for $0 < p < \alpha$.

\begin{remark}
	\label{rmk:Ledoux--Talagrand}
	Other series representations of stable Banach-valued random variables are possible.
	In particular, \citet[Sections 5.1 and 5.2]{LedouxTalagrand:1991} use series with random coefficients coming from the jumps of a Poisson process and the spectral measure of the random vector, and provide estimates for the strong and weak $L^{p}(\Omega, \mathbb{P}; \UU)$ norms of the induced random variable.
\end{remark}

\section{Well-posedness of Bayesian inverse problems on quasi-Banach spaces}
\label{sec:well-posed_BIP}

This section establishes conditions for the BIP with an arbitrary prior $\mu_{0}$ to be well-posed in sense that, for each $y \in \YY$, the posterior distribution $\mu^{y}$ of $u | y$ is a well-defined probability measure on $\UU$ (Theorem \ref{thm:posterior_defined}), which changes continuously when either the observed data is perturbed to $y' \approx y$ (Theorem \ref{thm:perturbed_y}) or the misfit function is perturbed to $\Phi_{N} \approx \Phi$ (Theorem \ref{thm:perturbed_Phi}).
It is natural to seek robustness of the BIP to such perturbations:
a perturbation of $y$ to $y'$ may arise through observational error, whereas a perturbation of $\Phi$ to $\Phi_{N}$ may arise through a numerical approximation of the forward model (e.g.\ a PDE solution operator) $G$ by a numerical solution operator $G_{N}$.
As in the earlier works following \citet{Stuart:2010}, the mapping $y \mapsto \mu^{y}$ is shown to be $(\| \quark \|_{\YY}, \Hellinger)$-Lipschitz, and the convergence $\mu_{N}^{y} \to \mu^{y}$ in $\Hellinger$ inherits the same convergence rate as the convergence $\Phi_{N} \to \Phi$, so that the numerical analysis of the forward problem transfers to the BIP.

A notable feature of the results presented in this section --- like all well-posedness results in the style of \citet{Stuart:2010} --- is that a careful tradeoff of growth rates of $\Phi$ is necessary in order to ensure well-definedness and well-posedness of the Bayesian posterior measure $\mu^{y}$.
Indeed, this tradeoff is a desirable feature, since `good' behaviour of one growth rate can be used to compensate for `bad' behaviour of another.
In the case of a heavy-tailed prior $\mu_{0}$, this tradeoff can be a particularly delicate task, since the class of integrable functions may be quite small.

The results of this section are not particular to \emph{stable} heavy-tailed priors, and the relaxed regularity assumptions used here provide additional understanding of the previously-studied Gaussian and Besov cases, in which well-posedness holds even when $M_{1, r}(t) \to - \infty$ at a polynomial rate as $t \to \infty$.
The proof strategies used here are very similar to those of \citet[Section 4]{Stuart:2010} and \citet[Section 4]{DashtiStuart:2015}.
Although the results of \citet[Section 4]{DashtiStuart:2015} have a similar level of generality in terms of $\Phi$ (modulo continuity/measurability assumptions), they are only explicitly applied there to uniform, Gaussian, and Besov priors on Banach spaces. 
Thus, the stable case considered here broadens the set of applications and Example~\ref{eg:final_example} later on elucidates that logarithmic growth rates are appropriate for stable priors, cf.\ quadratic rates for Gaussian priors.
The relaxation of the usual assumption that $\UU$ and $\YY$ are Banach spaces to allow them to be quasi-Banach space appears to be new, even though it introduces no significant complications in the proof.

\newcommand{\Aref}[1]{\textup{(A\ref{#1})}}
\begin{assumption}
	$\UU$ and $\YY$ are separable quasi-Banach spaces over $\R$ and the misfit function $\Phi \colon \UU \times \YY \to \R$ satisfies the following:
	\begin{compactenum}
		\renewcommand{\labelenumi}{(A\arabic{enumi})}
		\setcounter{enumi}{-1}
		\item \label{BIP_A_0} $\Phi$ is a locally bounded Carath\'eodory function, i.e.\ $\Phi(u; \quark)$ is continuous for each $u \in \UU$, $\Phi(\quark; y)$ is measurable for each $y \in \YY$, and for every $r > 0$, there exists $M_{0, r} \in \R$ such that, for all $(u, y) \in \UU \times \YY$ with $\| u \|_{\UU} < r$ and $\| y \|_{\YY} < r$,
		\[
			| \Phi(u; y) | \leq M_{0, r}.
		\]
		
		\item \label{BIP_A_1} For every $r > 0$, there exists a measurable $M_{1, r} \colon \R_{+} \to \R$ such that, for all $(u, y) \in \UU \times \YY$ with $\| y \|_{\YY} < r$,
		\[
			\Phi(u; y) \geq M_{1, r} \bigl( \| u \|_{\UU} \bigr) .
		\]

		\item \label{BIP_A_2} For every $r > 0$, there exists a measurable $M_{2, r} \colon \R_{+} \to \R_{+}$ such that, for all $(u, y_{1}, y_{2}) \in \UU \times \YY \times \YY$ with $\| y_{1} \|_{\YY} < r$, $\| y_{2} \|_{\YY} < r$,
		\[
			| \Phi(u; y_{1}) - \Phi(u; y_{2}) | \leq \exp \bigl( M_{2, r} \bigl( \| u \|_{\UU} \bigr) \bigr) \| y_{1} - y_{2} \|_{\YY} .
		\]
	\end{compactenum}
	Furthermore, for each $N \in \N$, $\Phi_{N} \colon \UU \times \YY \to \R$ is an approximation to $\Phi$ that satisfies \Aref{BIP_A_0}--\Aref{BIP_A_2} with $M_{i, r}$ independent of $N$, and such that
	\begin{compactenum}
		\renewcommand{\labelenumi}{(A\arabic{enumi})}
		\setcounter{enumi}{2}
		\item \label{BIP_A_3} $\Psi \colon \N \to \R_{+}$ is such that, for every $r > 0$, there exists a measurable $M_{3, r} \colon \R_{+} \to \R_{+}$, such that, for all $(u, y) \in \UU \times \YY$ with $\| y \|_{\YY} < r$,
		\[
			| \Phi_{N}(u; y) - \Phi(u; y) | \leq \exp \bigl( M_{3, r} \bigl( \| u \|_{\UU} \bigr) \bigr) \Psi(N) .
		\]
	\end{compactenum}
\end{assumption}

\begin{remark}
	\label{rmk:measurability}
	Assumptions \Aref{BIP_A_0}--\Aref{BIP_A_3} have been re-ordered relative to their counterparts in earlier works, such as those of \citet{Stuart:2010} and \citet{DashtiHarrisStuart:2012}.
	The numbering and placement of \Aref{BIP_A_0} (usually assumptions 2 and 3 in the previous works) highlights its role as a mild measurability assumption, so that \Aref{BIP_A_1}--\Aref{BIP_A_3} (usually assumptions 1, 4, and 5) form a natural sequence of statements about the growth rates $M_{i, r}$.
		
	\Aref{BIP_A_0} is weaker than the corresponding assumptions in previous works, in which it is assumed that $\Phi$ is locally Lipschitz continuous \citep[Assumption 2.6]{Stuart:2010} or continuous \citep[Assumptions 4.2]{DashtiStuart:2015}.
	However, close inspection of the proofs in those works reveals that continuity is used only in order to ensure that $e^{- \Phi(\quark; y)}$ is locally $\mu_{0}$-integrable, so that it can serve as a density of the non-normalised posterior with respect to the prior.
	The above assumptions imply that, $\Phi(\quark; y)$ and $e^{- \Phi(\quark; y)}$ are locally bounded measurable functions;
	since $\mu_{0}$ is a probability measure, this yields the desired local integrability.
	Furthermore, the separability assumptions on $\UU$ and $\YY$ and \Aref{BIP_A_0} imply, by \citet[Lemma 4.51]{AliprantisBorder:2006}, that $\Phi(u; y)$ and $e^{- \Phi(u; y)}$ are jointly measurable in $(u, y)$.
	
	However, \Aref{BIP_A_2} remains as a continuity assumption, since this is necessary in order to establish Hellinger continuity of the posterior with respect to $y$.
\end{remark}

\begin{theorem}[Well-definedness of the Bayesian posterior]
	\label{thm:posterior_defined}
	Let $\mu_{0} \in \mathcal{M}_{1}(\UU)$ be a Borel probability measure, and let $y \in \YY$.
	If \Aref{BIP_A_0} and \Aref{BIP_A_1} hold with
	\begin{equation}
		\label{eq:integrability_assumption}
		S_{1, r} \defeq \E_{u \sim \mu_{0}} \bigl[ \exp ( - M_{1, r} (\| u \|_{\UU}) ) \bigr] < \infty ,
	\end{equation}
	then $Z(y) \defeq \E_{u \sim \mu_{0}} \bigl[ \exp( - \Phi(u; y)) \bigr]$ is strictly positive and finite, and setting
	\begin{equation}
		\label{eq:posterior_prior_density}
		\frac{\rd \mu^{y}}{\rd \mu_{0}} (u) = \frac{ \exp( - \Phi(u; y)) }{Z(y)}
	\end{equation}
	defines a Borel probability measure $\mu^{y}$ on $\UU$, which is tight\footnote{Tightness is also referred to as being a \emph{Radon measure}.} in the sense that
	\[
		\mu^{y}(E) = \sup \{ \mu^{y}(K) \mid K \subseteq E \text{ and $K$ is compact} \}
		\quad
		\text{for all measurable $E \subseteq \UU$.}
	\]
\end{theorem}

\begin{proof}
	As discussed in Remark \ref{rmk:measurability}, $\exp( - \Phi(\quark; y) )$ is locally integrable with respect to $\mu_{0}$.
	Therefore, by the Radon--Nikod\'ym theorem, setting
	\[
		\nu(E) \defeq \E_{u \sim \mu_{0}} \bigl[ \exp(- \Phi(u; y)) \one[u \in E] \bigr]
	\]
	for each measurable set $E \subseteq \UU$ defines a countably additive measure $\nu$ on $\UU$;
	what remains is to check that $\nu$ can be normalised to yield the probability measure $\mu^{y}$, i.e.\ it is necessary to show that $0 < Z(y) \equiv \nu(\UU) < \infty$.
	Let $r > \| y \|_{\YY}$.
	Then
	\[
		Z(y) 
		\leq \E_{u \sim \mu_{0}} \bigl[ \exp \bigl( - M_{1, r} \bigl( \| u \|_{\UU} \bigr) \bigr) \bigr] 
		\leq S_{1, r} < \infty 
	\]
	by \Aref{BIP_A_1} and \eqref{eq:integrability_assumption};
	and
	\[
		Z(y) 
		\geq \E_{u \sim \mu_{0}} \bigl[ \exp( - \Phi(u; y)) \one \bigl[ \| u \|_{\UU} < r \bigr] \bigr] 
		\geq \exp( - M_{0, r} ) \mu_{0} \bigl( \Ball_{r}(0; \| \quark \|_{\UU}) \bigr) 
	\]
	by \Aref{BIP_A_0}.
	Since $\mu_{0}$ is a countably additive Borel probability measure,
	\[
		1 
		= \mu_{0}(\mathcal{U}) 
		= \sum_{n \in \N} \mu_{0} \bigl( \Ball_{n + 1}(0; \| \quark \|_{\mathcal{U}} ) \setminus \Ball_{n}(0; \| \quark \|_{\mathcal{U}} ) \bigr) ,
	\]
	and so it is impossible for all the summands on the right-hand side to vanish.
	Since at least one of the annuli $\Ball_{n + 1}(0; \| \quark \|_{\mathcal{U}} ) \setminus \Ball_{n}(0; \| \quark \|_{\mathcal{U}} )$ has strictly positive measure, it follows that $\mu_{0} \bigl( \Ball_{r}(0; \| \quark \|_{\UU} \bigr) > 0$ once $r > 0$ is large enough.
	Hence, $Z(y) > 0$, and so $\mu^{y}$ is a well-defined Borel probability measure on $\UU$, with Radon--Nikod\'ym derivative with respect to $\mu_{0}$ given by \eqref{eq:posterior_prior_density}.
	
	In any Polish space, and hence in the separable quasi-Banach space $\UU$, every finite-mass measure is tight \citep[Theorem 12.7]{AliprantisBorder:2006}.
	Hence, $\mu_{0}$ and $\mu^{y}$ are both tight.
\end{proof}

\begin{theorem}[Perturbation of observed data]
	\label{thm:perturbed_y}
	Suppose that $r > 0$ is such that \Aref{BIP_A_0}--\Aref{BIP_A_2} hold with
	\begin{equation}
		\label{eq:perturbed_y_assumption}
		S_{1, 2, r} \defeq \E_{u \sim \mu_{0}} \bigl[ \exp ( 2 M_{2, r} ( \| u \|_{\UU} ) - M_{1, r} ( \| u \|_{\UU} ) ) \bigr] < \infty.
	\end{equation}
	Then there exists a constant $C$, which may depend on $r$, $S_{1, 2, r}$, and the constants and functions in \Aref{BIP_A_0}--\Aref{BIP_A_2}, such that, whenever $\| y \|_{\YY}, \| y' \|_{\YY} < r$,
	\begin{align}
		\label{eq:Z_Lipschitz}
		| Z(y) - Z(y') | &\leq C \| y - y' \|_{\YY} \\
		\label{eq:dH_Lipschitz}
		\text{and} \quad \quad \Hellinger\bigl( \mu^{y}, \mu^{y'} \bigr) & \leq C \| y - y' \|_{\YY}.
	\end{align}
\end{theorem}

\begin{remark}
	By Kraft's inequality \citep{Kraft:1955, Steerneman:1983}, the assumptions of Theorem \ref{thm:perturbed_y} also imply well-posedness on the total variation metric:
	\[
		\sup \left\{ \bigl| \mu^{y}(E) - \mu^{y'}(E) \bigr| \,\middle|\, \text{Borel-measurable } E \subseteq \UU \right\} \leq \Hellinger\bigl( \mu^{y}, \mu^{y'} \bigr) \leq C \| y - y' \|_{\YY} .
	\]
\end{remark}

\begin{proof}[Proof of Theorem \ref{thm:perturbed_y}.]
	First, consider the normalising constant $Z(y)$ as a function of $y$.
	Note that \eqref{eq:perturbed_y_assumption} implies \eqref{eq:integrability_assumption}, so $0 < Z(y) < \infty$.
	Furthermore, whenever $\| y \|_{\YY}, \| y' \|_{\YY} < r$,
	\begin{align*}
		| Z(y) - Z(y') |
		& \leq \E_{u \sim \mu_{0}} \Bigl[ \bigl| \exp( - \Phi(u; y) ) - \exp ( - \Phi(u; y') ) \bigr| \Bigr] \\
		& \leq \E\Bigl[ e^{ - M_{1, r} ( \| u \|_{\UU} ) } \bigl| \Phi(u; y) - \Phi(u; y') \bigr| \Bigr] & & \text{by \Aref{BIP_A_1}} \\
		& \leq \E\Bigl[ e^{ - M_{1, r} ( \| u \|_{\UU} ) } e^{ M_{2, r} ( \| u \|_{\UU} ) } \| y - y' \|_{\YY} \Bigr] & & \text{by \Aref{BIP_A_2}} \\
		& \leq C \| y - y' \|_{\YY} & & \text{by \eqref{eq:perturbed_y_assumption},}
	\end{align*}
	which establishes \eqref{eq:Z_Lipschitz}.
	Now, from the definition \eqref{eq:Hellinger} of $\Hellinger$,
	\begin{align*}
		&\Hellinger\bigl( \mu^{y}, \mu^{y'} \bigr)^{2} \\
		&\quad= \E_{u \sim \mu_{0}}  \Biggl[ \biggl| \frac{e^{- \Phi(u; y) / 2}}{\sqrt{Z(y)}} - \frac{e^{- \Phi(u; y') / 2}}{\sqrt{Z(y')}} \biggr|^{2} \Biggr] \\
		&\quad= \E \Biggl[ \biggl| \frac{e^{- \Phi(u; y) / 2}}{\sqrt{Z(y)}} - \frac{e^{- \Phi(u; y') / 2}}{\sqrt{Z(y)}} + \frac{e^{- \Phi(u; y') / 2}}{\sqrt{Z(y)}} - \frac{e^{- \Phi(u; y') / 2}}{\sqrt{Z(y')}} \biggr|^{2} \Biggr] \\
		&\quad\leq \frac{2}{Z(y)} \E \Bigl[ \bigl| e^{- \Phi(u; y) / 2} - e^{- \Phi(u; y') / 2} \bigr|^{2} \Bigr] + 2 \E \Biggl[ e^{- \Phi(u; y')} \biggl| \frac{1}{\sqrt{Z(y)}} - \frac{1}{\sqrt{Z(y')}} \biggr|^{2} \Biggr] \\ 
		&\quad= \frac{2}{Z(y)} \underbrace{ \E \Bigl[ \bigl| e^{- \Phi(u; y) / 2} - e^{- \Phi(u; y') / 2} \bigr|^{2} \Bigr] }_{I_{1}} + 2 Z(y') \underbrace{ \biggl| \frac{1}{\sqrt{Z(y)}} - \frac{1}{\sqrt{Z(y')}} \biggr|^{2} }_{I_{2}} ,
	\end{align*}
	where the inequality follows from the algebraic inequality $(a + b)^{2} \leq 2 a^{2} + 2 b^{2}$.
	For the first term, 
	\begin{align*}
		I_{1}
		&\leq \E \bigl[ e^{ - M_{1, r} ( \| u \|_{\UU} ) } e^{ 2 M_{2, r} ( \| u \|_{\UU} ) } \| y - y' \|_{\YY}^{2} \bigr] & & \text{by \Aref{BIP_A_1}--\Aref{BIP_A_2}} \\
		&\leq C \| y - y' \|_{\YY}^{2} & & \text{by \eqref{eq:perturbed_y_assumption}.}
	\end{align*}
	For the second term, \eqref{eq:Z_Lipschitz} implies that
	\[
		I_{2}
		\leq \max \{ Z(y)^{-3}, Z(y')^{-3} \} \bigl| Z(y) - Z(y') \bigr|^{2}
		\leq C \| y - y' \|_{\YY}^{2}
	\]
	Thus, $\Hellinger\bigl( \mu^{y}, \mu^{y'} \bigr)^{2} \leq C \| y - y' \|_{\YY}^{2}$, and taking square roots completes the proof.
\end{proof}

\begin{theorem}[Perturbation of likelihood]
	\label{thm:perturbed_Phi}
	Let $\Phi$ and $\Phi_{N}$ satisfy \Aref{BIP_A_0}--\Aref{BIP_A_3}, and suppose that, for some $r > 0$,
	\begin{equation}
		\label{eq:perturbed_Phi_assumption}
		S_{1, 3, r} \defeq \E_{u \sim \mu_{0}} \bigl[ \exp ( 2 M_{3, r} ( \| u \|_{\UU} ) - M_{1, r} ( \| u \|_{\UU} ) ) \bigr] < \infty.
	\end{equation}
	Then there exists a constant $C$, which may depend on $r$, $S_{1, 3, r}$, and the constants and functions in \Aref{BIP_A_0}--\Aref{BIP_A_3} but is independent of $N$, such that the posteriors $\mu^{y}$ and $\mu^{y}_{N}$, arrived at using the same data $y$ with $\| y \|_{\YY} < r$ but the misfit functions $\Phi$ and $\Phi_{N}$ respectively, satisfy
	\[
		\Hellinger \bigl( \mu^{y}, \mu^{y}_{N} \bigr) \leq C \Psi(N).
	\]
\end{theorem}

\begin{proof}
	The proof is very similar to that of Theorem \ref{thm:perturbed_y}, and is omitted.
\end{proof}

\begin{remark}
	It is interesting to note the range of applicability of Theorems \ref{thm:posterior_defined}, \ref{thm:perturbed_y}, and \ref{thm:perturbed_Phi} when the prior $\mu_{0}$ is the probability law of a $\UU$-valued $\alpha$-stable random variable.
	Under the assumption \eqref{eq:q-frame_upper}, Theorem \ref{thm:Lp_convergence} implies that \eqref{eq:integrability_assumption} is satisfied if $M_{1, r}(t) \geq C - p \log t$ for some constant $C$ and some $0 < p < \alpha$, i.e.\ $\Phi(\quark; y)$ is permitted to diverge to $- \infty$ at a logarithmic rate controlled by the index of stability of $\mu_{0}$.
	Similarly, \eqref{eq:perturbed_y_assumption} is satisfied if $2 M_{2, r}(t) - M_{1, r}(t) \leq C + p \log t$, and \eqref{eq:perturbed_Phi_assumption} is satisfied if $2 M_{3, r}(t) - M_{1, r}(t) \leq C + p \log t$.
\end{remark}

\begin{example}
	\label{eg:final_example}
	Consider the additive noise model \eqref{eq:additive_noise} with Gaussian noise $\eta \sim \Normal(0, \Sigma)$.
	Suppose that the following growth conditions hold for some strictly positive constants $\sigma^{\pm} > 0$ and some continuous $g^{\pm} \colon \R_{+} \to \R_{+}$:
	\begin{align*}
		g^{-} (\| u \|_{\UU}) & \leq \| G(u) \|_{\YY} \leq g^{+} (\| u \|_{\UU}) & & \text{for all $u \in \UU$,} \\
		\sigma^{-} \| y \|_{\YY} & \leq \| \Sigma^{-1} y \|_{\YY} \leq \sigma^{+} \| y \|_{\YY} & & \text{for all $y \in \YY$.}
	\end{align*}
	Note that $\Phi(u; y)$ is smooth in $(u, y)$, so \Aref{BIP_A_0} holds.
	In fact, for a better lower bound on $\Phi(u; y)$ when $\| y \|_{\YY} < r$, observe that, for a suitable choice of additive constant $C_{r} \in \R$, 
	\[
		\Phi(u; y) \geq C_{r} + \sigma^{-} g^{-} ( \| u \|_{\UU} )^{2} ,
	\]
	so \Aref{BIP_A_1} holds with $M_{1, r}(t) = \sigma^{-} g^{-} (t)^{2}$.
	Thus, Theorem \ref{thm:posterior_defined}  gives well-definedness of $\mu^{y}$ even in the presence of the trivial lower bound $g^{-}(t) \equiv 0$, for any prior $\mu_{0}$.	
	
	However, the analysis of well-posedness is more involved:
	\[
		\frac{\partial \Phi}{\partial y} (u; y) = \Sigma^{-1} (y - G(u))
	\]
	and so $\sigma^{+} ( r + g^{+} (\| u \|_{\UU}) )$ is a Lipschitz constant for $\Phi(u; \quark)$ on $\Ball_{r}(0; \| \quark \|_{\YY}) \subseteq \YY$.
	Thus, \Aref{BIP_A_2} holds with $M_{2, r}(t) = \log (r + g^{+} (t))$.
	By Theorem \ref{thm:perturbed_y}, the posterior $\mu^{y}$ depends in a Lipschitz fashion on $y \in \Ball_{r}(0; \| \quark \|_{\YY}) \subseteq \YY$ if
	\[
		u \mapsto 2 \log (r + g^{+} (\| u \|_{\UU})) - \sigma^{-} g^{-} (\| u \|_{\UU})^{2}
	\]
	is exponentially integrable with respect to $\mu_{0}$.
	Suppose that the prior $\mu_{0}$ is the law of an $\Stable(\alpha, \bbeta, \ggamma, \ddelta; 0)$ random vector satisfying the assumptions of Theorem \ref{thm:Lp_convergence}.
	Then the hypotheses of Theorem \ref{thm:perturbed_y} are satisfied if, for some $p < \alpha$,
	\begin{equation}
		\label{eq:final_example_condition}
		2 \log (r + g^{+} (t)) - \sigma^{-} g^{-} (t)^{2} \leq C + p \log t .
	\end{equation}
	Since $p < 2$, the satisfaction of condition \eqref{eq:final_example_condition} depends crucially on the behaviour of $g^{\pm} (t)$ as $t \to \infty$.
	For example, suppose that the following slowly-growing lower bound and power-law upper bound on $\| G(u) \|_{\YY}$ hold:
	\[
		\sqrt{ c^{-} \log \| u \|_{\UU} } \leq \| G(u) \|_{\YY} \leq c^{+} \| u \|_{\UU}^{\kappa}.
	\]
	Then the BIP is well-posed with respect to $y$ if $2 \kappa - \sigma^{-} c^{-} \leq p$.
	Informally, this holds if the lower bounds on $\Sigma^{-1}$ and $G$ are far enough from zero compared to the upper bound growth rate $\kappa$.
	
	As usual, similar arguments apply to approximation of $\Phi$ by $\Phi_{N}$.
\end{example}

\begin{remark}
	The question of whether or not $\mu^{y}$ depends continuously upon the prior measure $\mu_{0}$ is a delicate one.
	First, probability measures on infinite-dimensional spaces are highly prone to mutual singularity even when they are related by surprisingly simple operations such as translation or dilation, cf.\ the Cameron--Martin and Feldman--H\'ajek theorems.
	Secondly, it is known that small perturbations of $\mu_{0}$ in the weak, total variation, or Hellinger topologies can lead to discontinuous changes in posterior expected values of pre-chosen integrands. 
	On the other hand, at least for finite-dimensional $\UU$, with respect to the Kullback--Leibler topology, small perturbations in $\mu_{0}$ lead to small perturbations in $\mu^{y}$. 
	For a more thorough treatment of this highly involved topic, see e.g.\ \citet[Section 1]{OwhadiScovel:2016} and the references cited therein.
\end{remark}

\section*{Acknowledgements}
\addcontentsline{toc}{section}{Acknowledgements}

The author is supported by the Free University of Berlin within the Excellence Initiative of the German Research Foundation, and thanks D.\ S.\ McCormick, H.\ C. Lie, J.\ Skilling, and A.\ M.\ Stuart for stimulating discussions, and two anonymous referees whose comments led to substantial improvements.

\bibliographystyle{abbrvnat}
\bibliography{./refs.bib}
\addcontentsline{toc}{section}{References}

\vfill

\end{document}